\documentclass[11pt,a4paper]{article}
\usepackage[affil-it]{authblk}
\usepackage[utf8]{inputenc}
\usepackage{amsmath, amssymb, mathrsfs, amsthm, amsfonts}
\usepackage{latexsym}
\usepackage[pdftex]{hyperref}
\usepackage{a4wide}
\usepackage{color}

\newtheorem{theorem}{Theorem}[section]
\newtheorem{corollary}[theorem]{Corollary}
\newtheorem{lemma}[theorem]{Lemma}
\newtheorem{proposition}[theorem]{Proposition}

\theoremstyle{definition}

\theoremstyle{remark}

\title{Diagonalization of Fix-Mahonian Matrices}

\author{Hery Randriamaro
\thanks{This research was funded by my mother \\
Lot II B 32 bis Faravohitra, 101 Antananarivo, Madagascar \\
e-mail: \texttt{hery.randriamaro@outlook.com}}}

\begin{document}

\maketitle

\begin{abstract}
\noindent Consider the regular representation of the sum over all permutations weighted by the sum of their descent, inversion, and fixed point multinomials. We compute the spectrum and the multiplicities of its elements of that matrix. Note that those multinomial statistics allow to apply the result on several permutation statistics like the number of fixed points, of descents, of inversions, and the major index at the same time.
	
\bigskip 
	
	
\smallskip
	
\end{abstract}

\section{Introduction}

\noindent Let $\mathfrak{S}_n$ be the symmetric group of $[n]$. Recall that the sets of descents, inversions, and fixed points of a permutation $\sigma \in \mathfrak{S}_n$ are respectively
\begin{itemize}
	\item $\mathrm{Des}\,\sigma := \big\{i \in [n-1]\ \big|\ \sigma(i)>\sigma(i+1)\big\}$,
	\item $\mathrm{Inv}\,\sigma := \big\{(i,j) \in [n]^2\ \big|\ i<j,\, \sigma(i)>\sigma(i+1)\big\}$,
	\item $\mathrm{Fix}\,\sigma := \big\{i \in [n]\ \big|\ \sigma(i)=i\big\}$.
\end{itemize}
We mainly study the statistics $\mathrm{des}_x: \mathfrak{S}_n \rightarrow \mathbb{R}[x_1, \dots, x_{n-1}]$, $\mathrm{inv}_y: \mathfrak{S}_n \rightarrow \mathbb{R}[y_{1,2}, \dots, y_{n,n-1}]$, and $\mathrm{fix}_z: \mathfrak{S}_n \rightarrow \mathbb{R}[z]$ respectively defined by
$$\mathrm{des}_x\,\sigma := \sum_{i \in \mathrm{Des}\,\sigma} x_i,\quad \mathrm{inv}_y\,\sigma := \sum_{(i,j) \in \mathrm{Inv}\,\sigma} y_{i,j},\quad \text{and} \quad \mathrm{fix}_z\,\sigma := \#\mathrm{Fix}(\sigma)\,z.$$
Those are obviously multinomial generalizations of the three classical combinatorial parameters, the number of descents, inversions, and fixed points. That latter is probably the oldest among those permutation statistics as it may originate from the problem of coincidences in the game of thirteen of De Montmort in 1708 \cite[p~185]{Mo}. It is often encountered in probability \cite{DiFuGu}. One also notes an appearance of the number of inversions already in 1888 when Laisant used it to classify the permutations \cite{La}. The number of descents is naturally associated to the descent algebra of $\mathfrak{S}_n$ \cite{BiPf}. But it also plays a key role in the combinatorial analysis of genome rearrangement \cite{BoPe}. In combinatorics, one usually meets those three statistics in distribution problems \cite{FoHa1}, \cite{FoHa2}, \cite{GaGe}. In this article, we study them in matrix context. Define the matrix 
$$\mathsf{IF}(n) := \big(\mathrm{inv}_y(\sigma \tau^{-1}) + \mathrm{fix}_z(\sigma \tau^{-1})\big)_{\sigma, \tau \in \mathfrak{S}_n}.$$ Let $\mathrm{Sp}(\mathsf{M})$ be the spectrum of a square matrix $\mathsf{M}$, and $\mathrm{m}_{\mathsf{M}}(x)$ the multiplicity of the eigenvalue $x \in \mathrm{Sp}(\mathsf{M})$. \emph{We aim to determine $\mathrm{Sp}\big(\mathsf{IF}(n)\big)$, and the multiplicity of each element in $\mathrm{Sp}\big(\mathsf{IF}(n)\big)$}. Computing with SageMath, we obtain 
\begin{itemize}
\item[$(n=1)$] $\mathrm{Sp}\big(\mathsf{IF}(1)\big) = \{z\}$ and $\mathrm{m}_{\mathsf{IF}(1)}(z) = 1$,
\item[$(n=2)$] $\mathrm{Sp}\big(\mathsf{IF}(2)\big) = \{y_{1,2} + 2z,\, -y_{1,2} + 2z\}$ and $\mathrm{m}_{\mathsf{IF}(2)}(y_{1,2} + 2z) = 1$, $\mathrm{m}_{\mathsf{IF}(2)}(-y_{1,2} + 2z) = 1$,
\item[$(n=3)$] $\mathrm{Sp}\big(\mathsf{IF}(3)\big) = \{3 y_{1,2} + 3 y_{1,3} + 3 y_{2,3} + 6z,\, - y_{1,2} +  y_{1,3} - y_{2,3},\, - y_{1,2} - 2y_{1,3} - y_{2,3} + 3z,\, 3z\}$,
\begin{itemize}
\item[$\bullet$] $\mathrm{m}_{\mathsf{IF}(3)}(3 y_{1,2} + 3 y_{1,3} + 3 y_{2,3} + 6z) = 1$,
\item[$\bullet$] $\mathrm{m}_{\mathsf{IF}(3)}- y_{1,2} +  y_{1,3} - y_{2,3}) = 1$,
\item[$\bullet$] $\mathrm{m}_{\mathsf{IF}(3)}(- y_{1,2} - 2y_{1,3} - y_{2,3} + 3z) = 2$,
\item[$\bullet$] $\mathrm{m}_{\mathsf{IF}(3)}(3z) = 2$.
\end{itemize}	
\end{itemize}

\begin{theorem}  \label{ThDIF}
For $n \geq 4$, $\mathsf{IF}(n)$ is diagonalizable, and 
\begin{align*}
\mathrm{Sp}\big(\mathsf{IF}(n)\big) = \bigg\{ & \frac{n!}{2} \sum_{\substack{(i,j) \in [n]^2 \\ i<j}} y_{i,j} + n!z,\, (n-2)!nz - (n-2)! \sum_{\substack{(i,j) \in [n]^2 \\ i<j}} (j-i)y_{i,j}, \\
& -(n-3)! \sum_{\substack{(i,j) \in [n]^2 \\ i<j}} \big(n -2(j-i)\big) y_{i,j},\, (n-2)!nz,\, 0 \bigg\} \quad \text{with}
\end{align*}
\begin{itemize}
\item $\displaystyle \mathrm{m}_{\mathsf{IF}(n)}\bigg( \frac{n!}{2} \sum_{\substack{(i,j) \in [n]^2 \\ i<j}} y_{i,j} + n!z \bigg) = 1$,
\item $\displaystyle \mathrm{m}_{\mathsf{IF}(n)}\bigg( (n-2)!nz - (n-2)! \sum_{\substack{(i,j) \in [n]^2 \\ i<j}} (j-i)y_{i,j} \bigg) = n-1$,
\item $\displaystyle \mathrm{m}_{\mathsf{IF}(n)}\bigg( -(n-3)! \sum_{\substack{(i,j) \in [n]^2 \\ i<j}} \big(n -2(j-i)\big) y_{i,j} \bigg) = \binom{n-1}{2}$,
\item $\displaystyle \mathrm{m}_{\mathsf{IF}(n)}\big( (n-2)!nz \big) = (n-1)(n-2)$,
\item $\displaystyle \mathrm{m}_{\mathsf{IF}(n)}(0) = n! -\frac{n}{2}(3n-7) -3$.
\end{itemize}
\end{theorem}

\begin{corollary}  \label{CoDIF}
Let $n \geq 4$, and $\mathsf{DIF}(n) := \big(\mathrm{des}_x(\sigma \tau^{-1}) + \mathrm{inv}_y(\sigma \tau^{-1}) + \mathrm{fix}_z(\sigma \tau^{-1})\big)_{\sigma, \tau \in \mathfrak{S}_n}$. Then, $\mathsf{DIF}(n)$ is diagonalizable, and 
\begin{align*}
	\mathrm{Sp}\big(\mathsf{DIF}(n)\big) = \bigg\{ & \frac{n!}{2} \Big( \sum_{i \in [n-1]} x_i + \sum_{\substack{(i,j) \in [n]^2 \\ i<j}} y_{i,j} \Big) + n!z,\, -(n-2)! \Big( \sum_{i \in [n-1]} x_i + \sum_{\substack{(i,j) \in [n]^2 \\ i<j}} (j-i)y_{i,j} - nz \Big), \\
	& -(n-2)! \sum_{i \in [n-1]} x_i - (n-3)! \sum_{\substack{(i,j) \in [n]^2 \\ i<j}} \big(n -2(j-i)\big) y_{i,j},\, (n-2)!nz,\, 0 \bigg\} \quad \text{with}
\end{align*}
\begin{itemize}
	\item $\displaystyle \mathrm{m}_{\mathsf{DIF}(n)}\bigg( \frac{n!}{2} \Big( \sum_{i \in [n-1]} x_i + \sum_{\substack{(i,j) \in [n]^2 \\ i<j}} y_{i,j} \Big) + n!z \bigg) = 1$,
	\item $\displaystyle \mathrm{m}_{\mathsf{DIF}(n)}\bigg( -(n-2)! \Big( \sum_{i \in [n-1]} x_i + \sum_{\substack{(i,j) \in [n]^2 \\ i<j}} (j-i)y_{i,j} - nz \Big) \bigg) = n-1$,
	\item $\displaystyle \mathrm{m}_{\mathsf{DIF}(n)}\bigg( -(n-2)! \sum_{i \in [n-1]} x_i - (n-3)! \sum_{\substack{(i,j) \in [n]^2 \\ i<j}} \big(n -2(j-i)\big) y_{i,j} \bigg) = \binom{n-1}{2}$,
	\item $\displaystyle \mathrm{m}_{\mathsf{DIF}(n)}\big( (n-2)!nz \big) = (n-1)(n-2)$,
	\item $\displaystyle \mathrm{m}_{\mathsf{DIF}(n)}(0) = n! -\frac{n}{2}(3n-7) -3$.
\end{itemize}
\end{corollary}

\begin{proof}
We get $\mathsf{DIF}(n)$ from $\mathsf{IF}(n)$ by replacing $y_{i,i+1}$ with $x_i + y_{i,i+1}$ for every $i \in [n-1]$.
\end{proof}

\noindent Note that $\displaystyle \frac{\mathsf{DIF}(n)}{\frac{n!}{2} \Big( \sum_{i \in [n-1]} x_i + \sum_{\substack{(i,j) \in [n]^2 \\ i<j}} y_{i,j} \Big) + n!z}$ is a stochastic matrix. Corollary~\ref{CoDIF} is clearly a generalization of \cite[Theorem~1.1]{Ra1}, and of \cite[Theorem~1.4]{Ra1}. The multinomial statistics have the advantage of being able to consider several statistics at the same time. With $\mathrm{fix}\,\sigma := \#\mathrm{Fix}(\sigma)$, and the Mahonian statistics major index $\displaystyle \mathrm{maj}\,\sigma := \sum_{i \in \mathrm{Des}\,\sigma}i$ and number of inversions $\mathrm{inv}\,\sigma := \#\mathrm{Inv}\,\sigma$ for instance, we get the following result.

\begin{corollary}  
Let $n \geq 4$, and $\mathsf{Mif}(n) := \big(\mathrm{maj}(\sigma \tau^{-1}) + \mathrm{inv}(\sigma \tau^{-1}) + \mathrm{fix}(\sigma \tau^{-1})\big)_{\sigma, \tau \in \mathfrak{S}_n}$. Then, $\mathsf{Mif}(n)$ is diagonalizable, and
$$\mathrm{Sp}\big(\mathsf{Mif}(n)\big) = \Big\{n!\binom{n}{2}+n!,\, n(2-n)(n+5)\frac{(n-2)!}{6},\, -\frac{2n!}{6},\, n(n-2)!,\, 0 \Big\} \quad \text{with}$$
\begin{itemize}
	\item $\displaystyle \mathrm{m}_{\mathsf{Mif}(n)}\bigg(n!\binom{n}{2}+n!\bigg) = 1$,
	\item $\displaystyle \mathrm{m}_{\mathsf{Mif}(n)}\Big(n(2-n)(n+5)\frac{(n-2)!}{6}\Big) = n-1$,
	\item $\displaystyle \mathrm{m}_{\mathsf{Mif}(n)}\Big(-\frac{2n!}{3}\Big) = \binom{n-1}{2}$,
	\item $\displaystyle \mathrm{m}_{\mathsf{Mif}(n)}\big(n(n-2)!\big) = (n-1)(n-2)$,
	\item $\displaystyle \mathrm{m}_{\mathsf{Mif}(n)}(0) = n! -\frac{n}{2}(3n-7) -3$.
\end{itemize}
\end{corollary}

\begin{proof}
Set $x_i=i$, $y_{i,j}=1$, and $z=1$ in Corollary~\ref{CoDIF}.
\end{proof}

\noindent Furthermore, Renteln computed the eigenvalues and multiplicities of $\big(\mathrm{inv}(\sigma \tau^{-1})\big)_{\sigma, \tau \in \mathfrak{S}_n}$ as distance matrix of the Cayley graph of $\mathfrak{S}_n$ \cite[§~4.8]{Re}.

\smallskip

\noindent This article is structured as follows. Define the matrix $\mathsf{F}(n) := \big(\mathrm{fix}_z(\sigma \tau^{-1})\big)_{\sigma, \tau \in \mathfrak{S}_n}$. 

\begin{theorem}  \label{ThF}
For $n \geq 4$, $\mathsf{F}(n)$ is diagonalizable, and $\mathrm{Sp}\big(\mathsf{F}(n)\big) = \big\{n!z,\, n(n-2)!z,\, 0\big\}$ with
\begin{itemize}
\item $\mathrm{m}_{\mathsf{F}(n)}(n!z) = 1$,
\item $\mathrm{m}_{\mathsf{F}(n)}\big(n(n-2)!z\big) = (n-1)^2$,
\item $\mathrm{m}_{\mathsf{F}(n)}(0) = n! - (n-1)^2 -1$.
\end{itemize} 
\end{theorem}

\noindent We prove Theorem~\ref{ThF} in Section~\ref{SeF}. Define $\mathsf{i}_y(n), \mathsf{f}_z(n) \in \mathbb{R}[y_{1,2}, \dots, y_{n,n-1},z][\mathfrak{S}_n]$ by
$$\mathsf{i}_y(n) := \sum_{\sigma \in \mathfrak{S}_n} \mathrm{inv}_y(\sigma)\sigma \quad \text{and} \quad \mathsf{f}_z(n) := \sum_{\sigma \in \mathfrak{S}_n} \mathrm{fix}_z(\sigma)\sigma.$$
Note that since $\mathsf{f}_z(n)$ is in the center of $\mathbb{R}[y_{1,2}, \dots, y_{n,n-1},z][\mathfrak{S}_n]$, then $\mathsf{i}_y(n) \mathsf{f}_z(n) = \mathsf{f}_z(n) \mathsf{i}_y(n)$.

\noindent Let $\mathscr{X}_M: \mathbb{R}[y_{1,2}, \dots, y_{n,n-1},z][\mathfrak{S}_n] \rightarrow \mathbb{R}[y_{1,2}, \dots, y_{n,n-1},z]^{d \times d}$ be the matrix representation of $\mathbb{R}[y_{1,2}, \dots, y_{n,n-1},z][\mathfrak{S}_n]$ on a module $M \subseteq \mathbb{R}[y_{1,2}, \dots, y_{n,n-1},z][\mathfrak{S}_n]$ of degree $d$. Denote by $\mathrm{Par}(n)$ the set formed by the partitions of $n$. For $\lambda \in \mathrm{Par}(n)$, let $\mathfrak{S}_n^{\lambda}$ be the conjugacy class of $\mathfrak{S}_n$ associated to $\lambda$, and $S^{\lambda}$ the Specht module of degree $d_{\lambda}$ associated to $\lambda$.

\begin{theorem}  \label{ThSp}
For $n \geq 4$, $\mathscr{X}_{S^{(n-1,1)}}\big(\mathsf{i}_y(n) + \mathsf{f}_z(n)\big)$ is diagonalizable, and
$$\mathrm{Sp}\Big(\mathscr{X}_{S^{(n-1,1)}}\big(\mathsf{i}_y(n) + \mathsf{f}_z(n)\big)\Big) = \Big\{(n-2)! \big(nz - \sum_{\substack{(i,j) \in [n]^2 \\ i<j}} (j-i)y_{i,j}\big),\, n(n-2)!z\Big\} \quad \text{with}$$
\begin{itemize}
	\item $\displaystyle \mathrm{m}_{\mathscr{X}_{S^{(n-1,1)}}\big(\mathsf{i}_y(n) + \mathsf{f}_z(n)\big)}\Big((n-2)! \big(nz - \sum_{\substack{(i,j) \in [n]^2 \\ i<j}} (j-i)y_{i,j}\big)\Big) = 1$,
	\item $\mathrm{m}_{\mathscr{X}_{S^{(n-1,1)}}\big(\mathsf{i}_y(n) + \mathsf{f}_z(n)\big)}\big(n(n-2)!z\big) = n-2$.
\end{itemize} 
\end{theorem}

\noindent We prove Theorem~\ref{ThSp} in Section~\ref{SeSp}. Then, we combine Theorem~\ref{ThF} and Theorem~\ref{ThSp} to prove Theorem~\ref{ThDIF} in Section~\ref{SeTh}. For $\lambda, \mu \in \mathrm{Par}(n)$, let $\chi_{\lambda}^{\mu}$ be the character associated to the Specht module $S^{\lambda}$ and to the conjugacy class $\mathfrak{S}_n^{\mu}$, and $\mathrm{fix}(\mu) := \mathrm{fix}(\sigma)$ for $\sigma \in \mathfrak{S}_n^{\mu}$. Besides, let $\mathsf{I}_n$ be the identity matrix of size $n$. It is known that \cite[§~2.2]{Ra}
\begin{equation} \label{ch}
\mathscr{X}_{\mathbb{R}[\mathfrak{S}_n]}\Big(\sum_{\sigma \in \mathfrak{S}_n} \mathrm{fix}(\sigma)\sigma\Big) = \bigoplus_{\lambda \in \mathrm{Par}(n)} \frac{\sum_{\mu \in \mathrm{Par}(n)}\mathrm{fix}(\mu) \cdot \#\mathfrak{S}_n^{\mu} \cdot \chi_{\lambda}^{\mu}}{d_{\lambda}} \mathsf{I}_{d_{\lambda}^2}.
\end{equation}

\noindent Although research on symmetric group characters remains very active \cite{Las}, \cite{OrZa}, it is still not possible to obtain Theorem~\ref{ThF} from Equation~\ref{ch} as the majority of the $\chi_{\lambda}^{\mu}$'s are not known. That legitimizes Theorem~\ref{ThF}. However, we deduce from Theorem~\ref{ThF} and Theorem~\ref{ThSp} that
$$\sum_{\mu \in \mathrm{Par}(n)}\mathrm{fix}(\mu) \cdot \#\mathfrak{S}_n^{\mu} \cdot \chi_{\lambda}^{\mu} = \begin{cases}
n! & \text{if}\ \lambda \in \big\{(1, \dots, 1), (n-1,1)\big\}, \\
0 & \text{otherwise}.
\end{cases}$$

\section{Diagonalization of $\mathsf{F}(n)$}  \label{SeF}

\noindent We prove Theorem~\ref{ThF} in this section. Our strategy is to compute the minimal polynomial of $\mathsf{F}(n)$. From it, we are able to deduce its diagonalizability, its spectrum, and to determine the multiplicities.

\begin{lemma}
For $n \geq 4$, $\displaystyle \sum_{\sigma \in \mathfrak{S}_n} \mathrm{fix}_z(\sigma) = n!z$.
\end{lemma}

\begin{proof}
Define $\mathrm{mfix}_z: \mathfrak{S}_n \rightarrow \mathbb{R}[z_1, \dots, z_n]$ by $\displaystyle \mathrm{mfix}_z(\sigma) := \sum_{i \in \mathrm{Fix}\,\sigma} z_i$, and let $$\mathrm{mf}(z_1, \dots, z_n) := \sum_{\sigma \in \mathfrak{S}_n} \mathrm{mfix}_z(\sigma).$$
Since $[z_i]\mathrm{mf}(z_1, \dots, z_n) = \#\big\{\sigma \in \mathfrak{S}_n\ |\ \sigma(i)=i\big\} = (n-1)!$, then $\mathrm{mf}(z_1, \dots, z_n)$ is equal to $\displaystyle (n-1)! \sum_{i \in [n]} z_i$. By $\displaystyle \sum_{\sigma \in \mathfrak{S}_n} \mathrm{fix}_z(\sigma) = \mathrm{mf}(z, \dots, z)$, we get the result.
\end{proof}

\begin{lemma}  \label{LeF}
Let $\sigma, \tau \in \mathfrak{S}_n$. Then, $\mathrm{Fix}(\sigma^{-1} \tau) = \big\{i \in [n]\ \big|\ \sigma(i) = \tau(i)\big\}$.
\end{lemma}

\begin{proof}
If $\sigma(i) = \tau(i)$, then $\sigma^{-1} \tau(i) = \sigma^{-1} \sigma(i) = i$. Otherwise, $\sigma^{-1} \tau(i) \neq \sigma^{-1} \sigma(i) = i$.
\end{proof}

\begin{lemma}  \label{Leij}
Let $n \geq 4$, and $i,j \in [n]$ with $i \neq j$. Then,
\begin{enumerate}
\item $\displaystyle \sum_{\substack{\sigma \in \mathfrak{S}_n \\ \sigma(i)=i,\, \sigma(j)=j}} \mathrm{fix}(\sigma) = 3(n-2)!$,
\item $\displaystyle \sum_{\substack{\sigma \in \mathfrak{S}_n \\ \sigma(i)=i,\, \sigma(j) \neq j}} \mathrm{fix}(\sigma) = (2n-5)(n-2)!$,
\item $\displaystyle \sum_{\substack{\sigma \in \mathfrak{S}_n \\ \sigma(i)=j,\, \sigma(j)=i}} \mathrm{fix}(\sigma) = (n-2)!$,
\item $\displaystyle \sum_{\substack{\sigma \in \mathfrak{S}_n \\ \sigma(i)=j,\, \sigma(j) \neq i}} \mathrm{fix}(\sigma) = (n-3)(n-2)!$.
\end{enumerate}
\end{lemma}

\begin{proof}
1. $\displaystyle \sum_{\substack{\sigma \in \mathfrak{S}_n \\ \sigma(i)=i,\, \sigma(j)=j}} \mathrm{fix}(\sigma) = \sum_{\substack{\sigma \in \mathfrak{S}_n \\ \sigma(n-1)=n-1,\, \sigma(n)=n}} \mathrm{fix}(\sigma) = 2(n-2)! + \sum_{\sigma \in \mathfrak{S}_{n-2}} \mathrm{fix}(\sigma) = 3(n-2)!$.

\noindent 2.
\begin{align*}
\sum_{\substack{\sigma \in \mathfrak{S}_n \\ \sigma(i)=i,\, \sigma(j) \neq j}} \mathrm{fix}(\sigma) & = \sum_{\substack{\sigma \in \mathfrak{S}_n \\ \sigma(n-1) \neq n-1,\, \sigma(n)=n}} \mathrm{fix}(\sigma) \\
& = (n-2)(n-2)! + \sum_{\substack{\sigma \in \mathfrak{S}_{n-1} \\ \sigma(n-1) \neq n-1}} \mathrm{fix}(\sigma) \\
& = (n-2)(n-2)! + (n-1)! - \sum_{\substack{\sigma \in \mathfrak{S}_{n-1} \\ \sigma(n-1) = n-1}} \mathrm{fix}(\sigma) \\
& = (2n-3)(n-2)! - (n-2)! - \sum_{\sigma \in \mathfrak{S}_{n-2}} \mathrm{fix}(\sigma) \\
& = (2n-5)(n-2)!.
\end{align*}

\noindent 3. $\displaystyle \sum_{\substack{\sigma \in \mathfrak{S}_n \\ \sigma(i)=j,\, \sigma(j)=i}} \mathrm{fix}(\sigma) = \sum_{\substack{\sigma \in \mathfrak{S}_n \\ \sigma(n-1)=n,\, \sigma(n)=n-1}} \mathrm{fix}(\sigma) = \sum_{\sigma \in \mathfrak{S}_{n-2}} \mathrm{fix}(\sigma) = (n-2)!$.

\noindent 4. For $i \in [n-2]$, we have $\displaystyle \sum_{\substack{\sigma \in \mathfrak{S}_n \\ \sigma(n)=i}} \mathrm{fix}(\sigma) = \sum_{\substack{\sigma \in \mathfrak{S}_n \\ \sigma(n)=n-1}} \mathrm{fix}(\sigma) = \mathsf{x}$. Then,
$$\sum_{\sigma \in \mathfrak{S}_n} \mathrm{fix}(\sigma) = \sum_{\substack{\sigma \in \mathfrak{S}_n \\ \sigma(n)=n}} \mathrm{fix}(\sigma) + (n-1) \mathsf{x} \quad \text{with} \quad \sum_{\substack{\sigma \in \mathfrak{S}_n \\ \sigma(n)=n}} \mathrm{fix}(\sigma) = 2(n-1)!,$$ 
which gives $\displaystyle \mathsf{x} = (n-2)(n-2)!$. Hence,
$$\sum_{\substack{\sigma \in \mathfrak{S}_n \\ \sigma(i)=j,\, \sigma(j) \neq i}} \mathrm{fix}(\sigma) = \sum_{\substack{\sigma \in \mathfrak{S}_n \\ \sigma(n-1) \neq n,\, \sigma(n)=n-1}} \mathrm{fix}(\sigma) = \mathsf{x} - \sum_{\substack{\sigma \in \mathfrak{S}_n \\ \sigma(n-1)=n,\, \sigma(n)=n-1}} \mathrm{fix}(\sigma) = (n-3)(n-2)!.$$
\end{proof}

\begin{lemma}  \label{Leijk}
Let $n \geq 4$, and $i,j,k \in [n]$ with $i \neq j$, $j \neq k$, $k \neq i$. Then,
\begin{enumerate}
	\item $\displaystyle \sum_{\substack{\sigma \in \mathfrak{S}_n \\ \sigma(i)=k,\, \sigma(j)=j}} \mathrm{fix}(\sigma) = (2n-5)(n-3)!$,
	\item $\displaystyle \sum_{\substack{\sigma \in \mathfrak{S}_n \\ \sigma(i)=k,\, \sigma(j) \neq j}} \mathrm{fix}(\sigma) = (n^2-6n+9)(n-3)!$,
	\item $\displaystyle \sum_{\substack{\sigma \in \mathfrak{S}_n \\ \sigma(i) \neq k,\, \sigma(j) = j}} \mathrm{fix}(\sigma) = (2n^2-8n+9)(n-3)!$,
	\item $\displaystyle \sum_{\substack{\sigma \in \mathfrak{S}_n \\ \sigma(i)=j,\, \sigma(j)=k}} \mathrm{fix}(\sigma) = (n-3)(n-3)!$,
	\item $\displaystyle \sum_{\substack{\sigma \in \mathfrak{S}_n \\ \sigma(i)=j,\, \sigma(j) \neq k}} \mathrm{fix}(\sigma) = (n^2-5n+7)(n-3)!$,
	\item $\displaystyle \sum_{\substack{\sigma \in \mathfrak{S}_n \\ \sigma(i) \neq j,\, \sigma(j) = k}} \mathrm{fix}(\sigma) = (n^2-5n+7)(n-3)!$.	
\end{enumerate}
\end{lemma}

\begin{proof}
1. We have $\displaystyle \sum_{\substack{\sigma \in \mathfrak{S}_n \\ \sigma(i)=k,\, \sigma(j)=j}} \mathrm{fix}(\sigma) = \sum_{\substack{\sigma \in \mathfrak{S}_n \\ \sigma(n-1)=n-2,\, \sigma(n)=n}} \mathrm{fix}(\sigma) = (n-2)! + \sum_{\substack{\sigma \in \mathfrak{S}_{n-1} \\ \sigma(n-1)=n-2}} \mathrm{fix}(\sigma)$.
From $\mathsf{x}$ in the proof 4. of Lemma~\ref{Leij}, we get $\displaystyle \sum_{\substack{\sigma \in \mathfrak{S}_{n-1} \\ \sigma(n-1)=n-2}} \mathrm{fix}(\sigma) = (n-3)(n-3)!$. Hence, $\displaystyle \sum_{\substack{\sigma \in \mathfrak{S}_n \\ \sigma(i)=k,\, \sigma(j)=j}} \mathrm{fix}(\sigma) = (n-2)! + (n-3)(n-3)! = (2n-5)(n-3)!$.

\noindent 2. On one side, $\displaystyle \sum_{\substack{\sigma \in \mathfrak{S}_n \\ \sigma(j) \neq j}} \mathrm{fix}(\sigma) = \sum_{l \in [n] \setminus \{j\}} \sum_{\substack{\sigma \in \mathfrak{S}_n \\ \sigma(j)=l}} \mathrm{fix}(\sigma) = (n-2)(n-1)!$. For $l,m \in [n] \setminus \{i,j\}$, we have $\displaystyle \sum_{\substack{\sigma \in \mathfrak{S}_n \\ \sigma(i)=l,\, \sigma(j) \neq j}} \mathrm{fix}(\sigma) = \sum_{\substack{\sigma \in \mathfrak{S}_n \\ \sigma(i)=m,\, \sigma(j) \neq j}} \mathrm{fix}(\sigma) = \mathsf{y}$. On the other side, $$\sum_{\substack{\sigma \in \mathfrak{S}_n \\ \sigma(j) \neq j}} \mathrm{fix}(\sigma) = \sum_{\substack{\sigma \in \mathfrak{S}_n \\ \sigma(i)=i,\, \sigma(j) \neq j}} \mathrm{fix}(\sigma) + \sum_{\substack{\sigma \in \mathfrak{S}_n \\ \sigma(i)=j}} \mathrm{fix}(\sigma) + (n-2)\mathsf{y}.$$
Hence, $\mathsf{y} = (n^2-6n+9)(n-3)!$.

\noindent 3. First $\displaystyle \sum_{\substack{\sigma \in \mathfrak{S}_n \\ \sigma(i) \neq k,\, \sigma(j) = j}} \mathrm{fix}(\sigma) = \sum_{\substack{\sigma \in \mathfrak{S}_n \\ \sigma(n-1) \neq n-2,\, \sigma(n) = n}} \mathrm{fix}(\sigma) = \sum_{\substack{\sigma \in \mathfrak{S}_{n-1} \\ \sigma(n-1) \neq n-2}} \mathrm{fix}(\sigma) + (n-2)(n-2)!$.\\ As $\displaystyle \sum_{\substack{\sigma \in \mathfrak{S}_{n-1} \\ \sigma(n-1) \neq n-2}} \mathrm{fix}(\sigma) = (n-1)! - \sum_{\substack{\sigma \in \mathfrak{S}_{n-1} \\ \sigma(n-1)=n-2}} \mathrm{fix}(\sigma) = (n^2-4n+6)(n-3)!$, therefore\\ $\displaystyle \sum_{\substack{\sigma \in \mathfrak{S}_n \\ \sigma(i) \neq k,\, \sigma(j) = j}} \mathrm{fix}(\sigma) = (2n^2-8n+9)(n-3)!$.

\noindent 4. For $l,m \in [n] \setminus \{i,j\}$, we have $\displaystyle \sum_{\substack{\sigma \in \mathfrak{S}_n \\ \sigma(i)=j,\, \sigma(j)=l}} \mathrm{fix}(\sigma) = \sum_{\substack{\sigma \in \mathfrak{S}_n \\ \sigma(i)=j,\, \sigma(j)=m}} \mathrm{fix}(\sigma) = \mathsf{z}$. Since
$$\sum_{\substack{\sigma \in \mathfrak{S}_n \\ \sigma(i)=j}} \mathrm{fix}(\sigma) = \sum_{\substack{\sigma \in \mathfrak{S}_n \\ \sigma(i)=j,\, \sigma(j)=i}} \mathrm{fix}(\sigma) + (n-2)\mathsf{z},$$
we get $\mathsf{z} = (n-3)(n-3)!$.

\noindent 5. As $\displaystyle \sum_{\substack{\sigma \in \mathfrak{S}_n \\ \sigma(i)=j}} \mathrm{fix}(\sigma) = \sum_{\substack{\sigma \in \mathfrak{S}_n \\ \sigma(i)=j,\, \sigma(j) \neq k}} \mathrm{fix}(\sigma) + \sum_{\substack{\sigma \in \mathfrak{S}_n \\ \sigma(i)=j,\, \sigma(j) = k}} \mathrm{fix}(\sigma)$, we deduce that
$$\sum_{\substack{\sigma \in \mathfrak{S}_n \\ \sigma(i)=j,\, \sigma(j) \neq k}} \mathrm{fix}(\sigma) = (n^2-5n+7)(n-3)!.$$

\noindent 6. As $\displaystyle \sum_{\substack{\sigma \in \mathfrak{S}_n \\ \sigma(j)=k}} \mathrm{fix}(\sigma) = \sum_{\substack{\sigma \in \mathfrak{S}_n \\ \sigma(i)=j,\, \sigma(j)=k}} \mathrm{fix}(\sigma) + \sum_{\substack{\sigma \in \mathfrak{S}_n \\ \sigma(i) \neq j,\, \sigma(j) = k}} \mathrm{fix}(\sigma)$, we deduce that
$$\sum_{\substack{\sigma \in \mathfrak{S}_n \\ \sigma(i) \neq j,\, \sigma(j)=k}} \mathrm{fix}(\sigma) = (n^2-5n+7)(n-3)!.$$
\end{proof}

\noindent Denote by $\iota$ the identity permutation of $\mathfrak{S}_n$.

\begin{proposition}  \label{Prfix}
For $n \geq 4$, and $\tau \in \mathfrak{S}_n$, we have
$$\sum_{\sigma \in \mathfrak{S}_n} \mathrm{fix}_z(\sigma) \mathrm{fix}_z(\sigma^{-1} \tau) = \sum_{\sigma \in \mathfrak{S}_n} \mathrm{fix}_z(\sigma) \mathrm{fix}_z(\sigma^{-1}) - n(n-2)! \big(\mathrm{fix}_z(\iota) - \mathrm{fix}_z(\tau)\big).$$
\end{proposition}

\begin{proof}
Consider first $\tau \in \mathfrak{S}_n$ such that $i,j \in \mathrm{Fix}\,\tau$ with $i \neq j$. From Lemma~\ref{LeF}, it is clear that $$\displaystyle \sum_{\substack{\sigma \in \mathfrak{S}_n \\ \sigma(i), \sigma(j) \notin \{i,j\}}} \mathrm{fix}(\sigma) \mathrm{fix}(\sigma^{-1} \tau) = \sum_{\substack{\sigma \in \mathfrak{S}_n \\ \sigma(i), \sigma(j) \notin \{i,j\}}} \mathrm{fix}(\sigma) \mathrm{fix}\big(\sigma^{-1} \tau (i\,j)\big).$$

\noindent If $\sigma(i)=i$ and $\sigma(j)=j$, then $\mathrm{fix}(\sigma^{-1} \tau) = \mathrm{fix}\big(\sigma^{-1} \tau (i\,j)\big) + 2$, and
$$\sum_{\substack{\sigma \in \mathfrak{S}_n \\ \sigma(i)=i,\, \sigma(j)=j}} \Big( \mathrm{fix}(\sigma) \mathrm{fix}(\sigma^{-1} \tau) - \mathrm{fix}(\sigma) \mathrm{fix}\big(\sigma^{-1} \tau (i\,j)\big) \Big) = 2 \sum_{\substack{\sigma \in \mathfrak{S}_n \\ \sigma(i)=i,\, \sigma(j)=j}} \mathrm{fix}(\sigma) = 6(n-2)!.$$

\noindent Similarly, using Lemma~\ref{Leij}, we obtain the following three equalities:

\begin{align*}
& \sum_{\substack{\sigma \in \mathfrak{S}_n \\ \sigma(i)=i,\, \sigma(j) \neq j}} \Big( \mathrm{fix}(\sigma) \mathrm{fix}(\sigma^{-1} \tau) - \mathrm{fix}(\sigma) \mathrm{fix}\big(\sigma^{-1} \tau (i\,j)\big) \Big) = (2n-5)(n-2)!, \\
& \sum_{\substack{\sigma \in \mathfrak{S}_n \\ \sigma(i)=j,\, \sigma(j) \neq i}} \Big( \mathrm{fix}(\sigma) \mathrm{fix}(\sigma^{-1} \tau) - \mathrm{fix}(\sigma) \mathrm{fix}\big(\sigma^{-1} \tau (i\,j)\big) \Big) = (3-n)(n-2)!, \\
& \sum_{\substack{\sigma \in \mathfrak{S}_n \\ \sigma(i)=j,\, \sigma(j)=i}} \Big( \mathrm{fix}(\sigma) \mathrm{fix}(\sigma^{-1} \tau) - \mathrm{fix}(\sigma) \mathrm{fix}\big(\sigma^{-1} \tau (i\,j)\big) \Big) = -2(n-2)!.
\end{align*}

\noindent After adding the five previous equations, we get at the end
\begin{equation}  \label{Eqij}
\sum_{\sigma \in \mathfrak{S}_n} \mathrm{fix}(\sigma) \mathrm{fix}\big(\sigma^{-1} \tau (i\,j)\big) = \sum_{\sigma \in \mathfrak{S}_n} \mathrm{fix}(\sigma) \mathrm{fix}(\sigma^{-1} \tau) - n(n-2)!\Big(\mathrm{fix}(\tau) - \mathrm{fix}\big(\tau (i\,j)\big)\Big).
\end{equation}

\noindent Consider now $\tau \in \mathfrak{S}_n$ such that $i_k \notin \mathrm{Fix}\,\tau$, $j \in \mathrm{Fix}\,\tau$, and $(i_1 \, \dots \, i_k)$ is a cycle of $\tau$. We have $$\displaystyle \sum_{\substack{\sigma \in \mathfrak{S}_n \\ \sigma(i_k), \sigma(j) \notin \{i_1,j\}}} \mathrm{fix}(\sigma) \mathrm{fix}(\sigma^{-1} \tau) = \sum_{\substack{\sigma \in \mathfrak{S}_n \\ \sigma(i_k), \sigma(j) \notin \{i_1,j\}}} \mathrm{fix}(\sigma) \mathrm{fix}\big(\sigma^{-1} \tau (i_k\,j)\big).$$

\noindent If $\sigma(i_k)=i_1$ and $\sigma(j)=j$, then $\mathrm{fix}(\sigma^{-1} \tau) = \mathrm{fix}\big(\sigma^{-1} \tau (i_k\,j)\big) + 2$, and
$$\sum_{\substack{\sigma \in \mathfrak{S}_n \\ \sigma(i_k)=i_1,\, \sigma(j)=j}} \Big( \mathrm{fix}(\sigma) \mathrm{fix}(\sigma^{-1} \tau) - \mathrm{fix}(\sigma) \mathrm{fix}\big(\sigma^{-1} \tau (i_k\,j)\big) \Big) = 2 \sum_{\substack{\sigma \in \mathfrak{S}_n \\ \sigma(i_k)=i_1,\, \sigma(j)=j}} \mathrm{fix}(\sigma) = 2(2n-5)(n-3)!.$$

\noindent Similarly, using Lemma~\ref{Leijk}, we obtain the following five equalities:

\begin{align*}
& \sum_{\substack{\sigma \in \mathfrak{S}_n \\ \sigma(i_k)=i_1,\, \sigma(j) \neq j}} \Big( \mathrm{fix}(\sigma) \mathrm{fix}(\sigma^{-1} \tau) - \mathrm{fix}(\sigma) \mathrm{fix}\big(\sigma^{-1} \tau (i_k\,j)\big) \Big) = (n^2-6n+9)(n-3)!, \\
& \sum_{\substack{\sigma \in \mathfrak{S}_n \\ \sigma(i_k) \neq i_1,\, \sigma(j)=j}} \Big( \mathrm{fix}(\sigma) \mathrm{fix}(\sigma^{-1} \tau) - \mathrm{fix}(\sigma) \mathrm{fix}\big(\sigma^{-1} \tau (i_k\,j)\big) \Big) = (2n^2-8n+9)(n-3)!, \\
& \sum_{\substack{\sigma \in \mathfrak{S}_n \\ \sigma(i_k)=j,\, \sigma(j) \neq i_1}} \Big( \mathrm{fix}(\sigma) \mathrm{fix}(\sigma^{-1} \tau) - \mathrm{fix}(\sigma) \mathrm{fix}\big(\sigma^{-1} \tau (i_k\,j)\big) \Big) = -(n^2-5n+7)(n-3)!, \\
& \sum_{\substack{\sigma \in \mathfrak{S}_n \\ \sigma(i_k) \neq j,\, \sigma(j)=i_1}} \Big( \mathrm{fix}(\sigma) \mathrm{fix}(\sigma^{-1} \tau) - \mathrm{fix}(\sigma) \mathrm{fix}\big(\sigma^{-1} \tau (i_k\,j)\big) \Big) = -(n^2-5n+7)(n-3)!, \\
& \sum_{\substack{\sigma \in \mathfrak{S}_n \\ \sigma(i_k)=j,\, \sigma(j)=i_1}} \Big( \mathrm{fix}(\sigma) \mathrm{fix}(\sigma^{-1} \tau) - \mathrm{fix}(\sigma) \mathrm{fix}\big(\sigma^{-1} \tau (i_k\,j)\big) \Big) = -2(n-3)(n-3)!.
\end{align*}

\noindent After adding the seven previous equations, we get at the end
\begin{equation}  \label{Eqijk}
\sum_{\sigma \in \mathfrak{S}_n} \mathrm{fix}(\sigma) \mathrm{fix}\big(\sigma^{-1} \tau (i\,j)\big) = \sum_{\sigma \in \mathfrak{S}_n} \mathrm{fix}(\sigma) \mathrm{fix}(\sigma^{-1} \tau) - n(n-2)!\Big(\mathrm{fix}(\tau) - \mathrm{fix}\big(\tau (i_k\,j)\big)\Big).
\end{equation}

\noindent For every $\tau \in \mathfrak{S}_n \setminus \{\iota\}$, there exists a sequence $(i_1\,j_1), \dots, (i_k\,j_k)$ of transpositions such that $\displaystyle \tau = \prod_{l \in [k]}^{\rightarrow} (i_l\,j_l)$. Assume that $\tau_0 = \iota$, and for $m \in [k]$, set $\displaystyle \tau_m = \prod_{l \in [m]}^{\rightarrow} (i_l\,j_l)$. Those transpositions can be chosen so that, for every $l \in [k]$, either $i_l, j_l \in \mathrm{Fix}\,\tau_{l-1}$ or $i_l \notin \mathrm{Fix}\,\tau_{l-1}$, $j_l \in \mathrm{Fix}\,\tau_{l-1}$.

\noindent Using Equation~\ref{Eqij} and Equation~\ref{Eqijk}, we get
\begin{align*}
\sum_{\sigma \in \mathfrak{S}_n} \big( \mathrm{fix}(\sigma) \mathrm{fix}(\sigma^{-1} \tau) - \mathrm{fix}(\sigma) \mathrm{fix}(\sigma^{-1}) \big) & = \sum_{l \in [k]} \sum_{\sigma \in \mathfrak{S}_n} \big( \mathrm{fix}(\sigma) \mathrm{fix}(\sigma^{-1} \tau_l) - \mathrm{fix}(\sigma) \mathrm{fix}(\sigma^{-1} \tau_{l-1}) \big) \\
& = \sum_{l \in [k]} n(n-2)!\big(\mathrm{fix}(\tau_l) - \mathrm{fix}(\tau_{l-1})\big) \\
& = n(n-2)!\big(\mathrm{fix}(\tau) - \mathrm{fix}(\iota)\big).
\end{align*}

\noindent We obtain the result by replacing $\mathrm{fix}$ with $\mathrm{fix}_z$.
\end{proof}

\begin{corollary}  \label{CoF}
Let $n \geq 4$, and $t$ a variable. The minimal polynomial of $\mathsf{F}(n)$ is $$t(t - n!z)\big(t - n(n-2)!z\big).$$
\end{corollary}

\begin{proof}
Using Proposition~\ref{Prfix}, we obtain
\begin{align*}
\sum_{\sigma \in \mathfrak{S}_n} \mathrm{fix}_z(\sigma) & \sigma \Big(\sum_{\sigma \in \mathfrak{S}_n} \mathrm{fix}_z(\sigma)\sigma - n(n-2)!z \iota\Big) = \sum_{\tau \in \mathfrak{S}_n} \sum_{\sigma \in \mathfrak{S}_n} \mathrm{fix}_z(\sigma) \mathrm{fix}_z(\tau) \sigma \tau - n(n-2)!z \sum_{\tau \in \mathfrak{S}_n} \mathrm{fix}_z(\tau)\tau  \\
& = \sum_{\tau \in \mathfrak{S}_n} \sum_{\sigma \in \mathfrak{S}_n} \mathrm{fix}_z(\sigma) \mathrm{fix}_z(\sigma^{-1}\tau)\tau - n(n-2)!z \sum_{\tau \in \mathfrak{S}_n} \mathrm{fix}_z(\tau)\tau  \\
& = \sum_{\tau \in \mathfrak{S}_n} \Big( \sum_{\sigma \in \mathfrak{S}_n} \mathrm{fix}_z(\sigma) \mathrm{fix}_z(\sigma^{-1}) - n(n-2)!z \big(\mathrm{fix}_z(\iota) - \mathrm{fix}_z(\tau)\big) \Big) \tau - n(n-2)!z \sum_{\tau \in \mathfrak{S}_n} \mathrm{fix}_z(\tau)\tau  \\
& = \sum_{\tau \in \mathfrak{S}_n} \big( \sum_{\sigma \in \mathfrak{S}_n} \mathrm{fix}_z(\sigma) \mathrm{fix}_z(\sigma^{-1}) - n(n-2)!z \mathrm{fix}_z(\iota) \big) \tau  \\
& = \sum_{\sigma \in \mathfrak{S}_n} \Lambda \sigma
\end{align*}
with $\displaystyle \Lambda = \sum_{\sigma \in \mathfrak{S}_n} \mathrm{fix}_z(\sigma) \mathrm{fix}_z(\sigma^{-1}) - n(n-2)!z \mathrm{fix}_z(\iota)$. Thus,
\begin{align*}
\sum_{\sigma \in \mathfrak{S}_n} \mathrm{fix}_z(\sigma)\sigma & \Big(\sum_{\sigma \in \mathfrak{S}_n} \mathrm{fix}_z(\sigma)\sigma - n(n-2)!z \iota\Big) \Big(\sum_{\sigma \in \mathfrak{S}_n} \mathrm{fix}_z(\sigma)\sigma - n!z \iota\Big) = \Lambda \sum_{\sigma \in \mathfrak{S}_n}\sigma \Big(\sum_{\sigma \in \mathfrak{S}_n} \mathrm{fix}_z(\sigma)\sigma - n!z \iota\Big) \\
& = \Lambda \Big(\sum_{\tau \in \mathfrak{S}_n} \sum_{\sigma \in \mathfrak{S}_n} \mathrm{fix}_z(\tau) \sigma \tau - n!z \sum_{\tau \in \mathfrak{S}_n} \tau\Big)  \\
& = \Lambda \Big(\sum_{\tau \in \mathfrak{S}_n} \sum_{\sigma \in \mathfrak{S}_n} \mathrm{fix}_z(\sigma^{-1} \tau) \tau - n!z \sum_{\tau \in \mathfrak{S}_n} \tau\Big)  \\
& = \Lambda \Big(\sum_{\tau \in \mathfrak{S}_n} n!z \tau - n!z \sum_{\tau \in \mathfrak{S}_n} \tau\Big)  \\
& = 0.
\end{align*}
Hence, the minimal polynomial of $\mathsf{F}(n)$ divides $t(t - n!z)\big(t - n(n-2)!z\big)$. Furthermore, as $\displaystyle \sum_{\sigma \in \mathfrak{S}_n} \mathrm{fix}_z(\sigma)\sigma \Big(\sum_{\sigma \in \mathfrak{S}_n} \mathrm{fix}_z(\sigma)\sigma - n!z \iota\Big) = \sum_{\tau \in \mathfrak{S}_n} \Big( \sum_{\sigma \in \mathfrak{S}_n} \mathrm{fix}_z(\sigma) \mathrm{fix}_z(\sigma^{-1}) - n(n-2)! \mathrm{fix}_z(\iota) - n(n-2)(n-2)! \mathrm{fix}_z(\tau) \Big) \tau \neq 0$ and $\displaystyle \Big(\sum_{\sigma \in \mathfrak{S}_n} \mathrm{fix}_z(\sigma)\sigma - n(n-2)!z \iota\Big) \Big(\sum_{\sigma \in \mathfrak{S}_n} \mathrm{fix}_z(\sigma)\sigma - n!z \iota\Big) = \sum_{\tau \in \mathfrak{S}_n} \Big( \sum_{\sigma \in \mathfrak{S}_n} \mathrm{fix}_z(\sigma) \mathrm{fix}_z(\sigma^{-1}) - n(n-2)! \mathrm{fix}_z(\iota) - n! \mathrm{fix}_z(\tau) \Big) \tau + n(n-2)!n!\iota \neq 0$, then the minimal polynomial of $\mathsf{F}(n)$ divides neither $t(t - n!z)$ nor $(t - n!z)\big(t - n(n-2)!z\big)$.
\end{proof}

\noindent We finally need the following lemma for the proof of Theorem~\ref{ThF}.

\begin{lemma}  \label{LeJ}
Let $n \geq 4$, and $p(z) \in \mathbb{R}[z] \setminus \{0\}$. Then, the regular representation $\mathsf{J}_{n, p(z)}$ of $\displaystyle \sum_{\sigma \in \mathfrak{S}_n} p(z) \sigma \in \mathbb{R}[z][\mathfrak{S}_n]$ is diagonalizable, and $$\mathrm{Sp}(\mathsf{J}_{n, p(z)}) = \big\{n!p(z),\, 0\big\} \quad \text{with} \quad \mathrm{m}_{\mathsf{J}_{n, p(z)}}\big(n!p(z)\big) = 1,\ \mathrm{m}_{\mathsf{J}_{n, p(z)}}(0) = n!-1.$$ 
\end{lemma}

\begin{proof}
The eigenspace of $\mathsf{J}_{n, p(z)}$ associated to the eigenvalue $n!p(z)$ is $\displaystyle \Big\langle \sum_{\sigma \in \mathfrak{S}_n}\sigma \Big\rangle$, and that associated to $0$ is $\langle \iota - \sigma \rangle_{\sigma \in \mathfrak{S}_n \setminus \{\iota\}}$.
\end{proof}

\noindent Denote by $\mathrm{diag}\,\mathsf{M}$ the diagonalized matrix of a square matrix $\mathsf{M}$. We can now proceed to the proof of Theorem~\ref{ThF}.

\begin{proof}
From Corollary~\ref{CoF}, we know that $\mathsf{F}(n)$ is diagonalizable, and $$\mathrm{Sp}\big(\mathsf{F}(n)\big) = \big\{n!z,\, n(n-2)!z,\, 0\big\}.$$
We deduce that $\mathsf{X}_n = \mathsf{F}(n) - n(n-2)!z \mathsf{I}_{n!}$ is a also diagonalizable, and $$\mathrm{Sp}(\mathsf{X}_n) = \big\{n(n-2)(n-2)!z,\, -n(n-2)!z,\, 0\big\} \quad \text{with}$$
\begin{itemize}
\item $\mathrm{m}_{\mathsf{X}_n}\big(n(n-2)(n-2)!z\big) = \mathrm{m}_{\mathsf{F}(n)}(n!z)$,
\item $\mathrm{m}_{\mathsf{X}_n}\big(-n(n-2)!z\big) = \mathrm{m}_{\mathsf{F}(n)}(0)$,
\item $\mathrm{m}_{\mathsf{X}_n}(0) = \mathrm{m}_{\mathsf{F}(n)}\big(n(n-2)!z\big)$.
\end{itemize}
Using the first calculation in the proof of Corollary~\ref{CoF} and Lemma~\ref{LeJ}, we have on one side $$\mathrm{diag}\, \mathsf{F}(n) \mathsf{X}_n = \Lambda \mathsf{I}_1 \oplus 0 \mathsf{I}_{n!-1}.$$
On the other side, as $\mathsf{F}(n)$ and $\mathsf{X}_n$ are simultaneously, then
\begin{align*}
\mathrm{diag}\, \mathsf{F}(n) \mathsf{X}_n & = \mathrm{diag}\,\mathsf{F}(n) \times \mathrm{diag}\, \mathsf{X}_n  \\
& = n!z \mathsf{I}_{\mathrm{m}_{\mathsf{F}(n)}(n!z)} \oplus n(n-2)!z \mathsf{I}_{\mathrm{m}_{\mathsf{F}(n)}(n(n-2)!z)} \oplus 0 \mathsf{I}_{\mathrm{m}_{\mathsf{F}(n)}(0)} \\
& \quad \times n(n-2)(n-2)!z \mathsf{I}_{\mathrm{m}_{\mathsf{F}(n)}(n!z)} \oplus 0 \mathsf{I}_{\mathrm{m}_{\mathsf{F}(n)}(n(n-2)!z)} \oplus -n(n-2)!z \mathsf{I}_{\mathrm{m}_{\mathsf{F}(n)}(0)}  \\
& = n(n-2)(n-2)!n!z^2 \mathsf{I}_{\mathrm{m}_{\mathsf{F}(n)}(n!z)} \oplus 0 \mathsf{I}_{\mathrm{m}_{\mathsf{F}(n)}(n(n-2)!z)} \oplus 0 \mathsf{I}_{\mathrm{m}_{\mathsf{F}(n)}(0)}.
\end{align*}
We deduce after comparison that $\mathrm{m}_{\mathsf{F}(n)}(n!z) = 1$.

\noindent As $\mathrm{tr}\,\mathsf{F}(n) = nn!z$, thus $$n!z \, \mathrm{m}_{\mathsf{F}(n)}(n!z) + n(n-2)!z \, \mathrm{m}_{\mathsf{F}(n)}\big(n(n-2)!z\big) + 0 \, \mathrm{m}_{\mathsf{F}(n)}(0) = nn!z,$$
and $\mathrm{m}_{\mathsf{F}(n)}\big(n(n-2)!z\big) = (n-1)^2$.

\noindent Finally, since $\mathsf{F}(n)$ is a $n! \times n!$-matrix, then
$$\mathrm{m}_{\mathsf{F}(n)}(n!z) + \mathrm{m}_{\mathsf{F}(n)}\big(n(n-2)!z\big) + \mathrm{m}_{\mathsf{F}(n)}(0) = n!,$$
and $\mathrm{m}_{\mathsf{F}(n)}(0) = n! - (n-1)^2 -1$.
\end{proof}

\noindent Remark that $\Lambda = n(n-2)(n-2)!n!z^2$ implies $\displaystyle \sum_{\sigma \in \mathfrak{S}_n} \mathrm{fix}(\sigma)^2 = n^2(n-2)!\big((n-2)(n-1)! +1\big)$.

\section{Diagonalization of $\mathscr{X}_{S^{(n-1,1)}}\big(\mathsf{i}_y(n) + \mathsf{f}_z(n)\big)$}  \label{SeSp}

\noindent We prove Theorem~\ref{ThSp} in this section. We first compute $\mathrm{diag}\,\mathscr{X}_{S^{(n-1,1)}}\big(\mathsf{f}_z(n)\big)$, and then $\mathrm{diag}\,\mathscr{X}_{S^{(n-1,1)}}\big(\mathsf{i}_y(n)\big)$. Recall that $\{\mathbf{1}-\mathbf{i}\}_{i \in [n] \setminus \{1\}}$ is a basis of the Specht module $S^{(n-1,1)}$.

\begin{proposition}  \label{Prfix2}
For $n \geq 4$, $\mathscr{X}_{S^{(n-1,1)}}\big(\mathsf{f}_z(n)\big) = n(n-2)!z\mathsf{I}_{n-1}$.
\end{proposition}

\begin{proof}
Let $i,j \in [n] \setminus \{1\}$ with $i \neq j$. Using Lemma~\ref{Leij}, we obtain
\begin{align*}
[\mathbf{1}-\mathbf{i}]\mathsf{f}_z(n)(\mathbf{1}-\mathbf{i}) & = \sum_{\substack{\sigma \in \mathfrak{S}_n \\ \sigma(1)=1,\, \sigma(i)=i}}\mathrm{fix}_z(\sigma) - \sum_{\substack{\sigma \in \mathfrak{S}_n \\ \sigma(1)=i,\, \sigma(i)=1}}\mathrm{fix}_z(\sigma) + \sum_{\substack{\sigma \in \mathfrak{S}_n \\ \sigma(1) \neq 1,\, \sigma(i)=i}}\mathrm{fix}_z(\sigma) - \sum_{\substack{\sigma \in \mathfrak{S}_n \\ \sigma(1)=i,\, \sigma(i) \neq 1}}\mathrm{fix}_z(\sigma)  \\
& = 3(n-2)!z - (n-2)!z + (2n-5)(n-2)!z - (n-3)(n-2)!z \\
& = n(n-2)!z.
\end{align*}

\noindent Using Lemma~\ref{Leijk}, we obtain
\begin{align*}
[\mathbf{1}-\mathbf{j}]\mathsf{f}_z(n)(\mathbf{1}-\mathbf{i}) & = \sum_{\substack{\sigma \in \mathfrak{S}_n \\ \sigma(1)=1,\, \sigma(i)=j}}\mathrm{fix}_z(\sigma) - \sum_{\substack{\sigma \in \mathfrak{S}_n \\ \sigma(1)=j,\, \sigma(i)=1}}\mathrm{fix}_z(\sigma) + \sum_{\substack{\sigma \in \mathfrak{S}_n \\ \sigma(1) \neq 1,\, \sigma(i)=j}}\mathrm{fix}_z(\sigma) - \sum_{\substack{\sigma \in \mathfrak{S}_n \\ \sigma(1)=j,\, \sigma(i) \neq 1}}\mathrm{fix}_z(\sigma)  \\
& = (2n-5)(n-3)!z - (n-3)(n-3)!z + (n^2-6n+9)(n-3)!z - (n^2-5n+7)(n-3)!z \\
& = 0.
\end{align*}
\end{proof}

\begin{lemma}  \label{LeYij}
Let $n \geq 4$, and $i,j \in [n] \setminus \{1\}$. Then,
\begin{align*}
& 1.\ \sum_{\substack{\sigma \in \mathfrak{S}_n \\ \sigma(1)=1,\, \sigma(i)=j}} \mathrm{inv}_y(\sigma) = & (n-j)(n-3)!\sum_{l \in [i-1] \setminus \{1\}}y_{l,i} + (j-2)(n-3)!\sum_{m \in [n] \setminus [i]}y_{i,m} \\
&& + \frac{(n-2)!}{2}\sum_{\substack{l,m \in [n] \setminus \{1,i\}}}y_{l,m}, \\
& 2.\ \sum_{\substack{\sigma \in \mathfrak{S}_n \\ \sigma(1)=j,\, \sigma(i)=1}} \mathrm{inv}_y(\sigma) = & (n-2)!y_{1,i} + (j-2)(n-3)!\sum_{m \in [n] \setminus \{1,i\}}y_{1,m} \\
&& + (n-2)(n-3)!\sum_{l \in [i-1] \setminus \{1\}}y_{l,i} + \frac{(n-2)!}{2}\sum_{\substack{l,m \in [n] \setminus \{1,i\}}}y_{l,m}, \\
& 3.\ \sum_{\substack{\sigma \in \mathfrak{S}_n \\ \sigma(1) \neq 1,\, \sigma(i)=j}} \mathrm{inv}_y(\sigma) = & (n-j)(n-2)!y_{1,i} + \binom{n-1}{2}(n-3)!\sum_{m \in [n] \setminus \{1,i\}}y_{1,m}  \\
&& + \big((j-1)(n-4) + j\big)(n-3)!\sum_{m \in [n] \setminus [i]}y_{i,m} \\
&& + (n-j)(n-3)(n-3)!\sum_{l \in [i-1] \setminus \{1\}}y_{l,i} + \frac{n-2}{2}(n-2)!\sum_{\substack{l,m \in [n] \setminus \{1,i\}}}y_{l,m}, \\
& 4.\ \sum_{\substack{\sigma \in \mathfrak{S}_n \\ \sigma(1)=j,\, \sigma(i) \neq 1}} \mathrm{inv}_y(\sigma) = & (j-2)(n-2)!y_{1,i} + \big((j-1)(n-4) + j\big)(n-3)!\sum_{m \in [n] \setminus \{1,i\}}y_{1,m}  \\
&& + \binom{n-2}{2}(n-3)!\sum_{l \in [i-1] \setminus \{1\}}y_{l,i} + \binom{n-1}{2}(n-3)!\sum_{m \in [n] \setminus [i]}y_{i,m}  \\
&& + \frac{n-2}{2}(n-2)!\sum_{\substack{l,m \in [n] \setminus \{1,i\}}}y_{l,m}.
\end{align*}
\end{lemma}

\begin{proof}
1. For $l,m \in [n] \setminus \{1,i\}$, $$\displaystyle [y_{l,m}] \sum_{\substack{\sigma \in \mathfrak{S}_n \\ \sigma(1)=1,\, \sigma(i)=j}} \mathrm{inv}_y(\sigma) = \#\big\{\sigma \in \mathfrak{S}_n\ \big|\ \sigma(1)=1,\, \sigma(i)=j,\, \sigma(l)<\sigma(m)\big\} = \frac{(n-2)!}{2}.$$
For $l \in [i-1] \setminus \{1\}$, $$\displaystyle [y_{l,i}] \sum_{\substack{\sigma \in \mathfrak{S}_n \\ \sigma(1)=1,\, \sigma(i)=j}} \mathrm{inv}_y(\sigma) = \#\big\{\sigma \in \mathfrak{S}_n\ \big|\ \sigma(1)=1,\, \sigma(i)=j,\, \sigma(l)>j\big\} = (n-j)(n-3)!.$$
For $m \in [n] \setminus [i]$, $$\displaystyle [y_{i,m}] \sum_{\substack{\sigma \in \mathfrak{S}_n \\ \sigma(1)=1,\, \sigma(i)=j}} \mathrm{inv}_y(\sigma) = \#\big\{\sigma \in \mathfrak{S}_n\ \big|\ \sigma(1)=1,\, \sigma(i)=j,\, j>\sigma(m)\big\} = (j-2)(n-3)!.$$
2. Like in 1., but we use $\#\big\{\sigma \in \mathfrak{S}_n\ \big|\ \sigma(1)=j,\, \sigma(i)=1\big\} = (n-2)!$,
\begin{align*}
& \text{for}\ m \in [n] \setminus \{1,i\},\ \#\big\{\sigma \in \mathfrak{S}_n\ \big|\ \sigma(1)=j,\, \sigma(i)=1,\, j>\sigma(m)\big\} = (j-2)(n-3)!, \\
& \text{for}\ l \in [i] \setminus \{1\},\ \#\big\{\sigma \in \mathfrak{S}_n\ \big|\ \sigma(1)=j,\, \sigma(i)=1,\, \sigma(l)>1\big\} = (n-2)(n-3)!.
\end{align*}
3. Like in 1., but we use $\#\big\{\sigma \in \mathfrak{S}_n\ \big|\ \sigma(1) \neq 1,\, \sigma(i)=j,\, \sigma(1)>j\big\} = (n-j)(n-2)!$,
\begin{align*}
& \text{for}\ m \in [n] \setminus \{1,i\},\ \#\big\{\sigma \in \mathfrak{S}_n\ \big|\ \sigma(1) \neq 1,\, \sigma(i)=j,\, \sigma(1)>\sigma(m)\big\} = \binom{n-1}{2}(n-3)!, \\
& \text{for}\ l \in [i-1] \setminus \{1\},\ \#\big\{\sigma \in \mathfrak{S}_n\ \big|\ \sigma(1) \neq 1,\, \sigma(i)=j,\, \sigma(l)>j\big\} = (n-j)(n-3)(n-3)!, \\
& \text{for}\ m \in [n] \setminus [i],\ \#\big\{\sigma \in \mathfrak{S}_n\ \big|\ \sigma(1) \neq 1,\, \sigma(i)=j,\, j>\sigma(m)\big\} = \big((j-1)(n-4) + j\big)(n-3)!, \\
& \text{for}\ l,m \in [n] \setminus \{1,i\},\ \#\big\{\sigma \in \mathfrak{S}_n\ \big|\ \sigma(1) \neq 1,\, \sigma(i)=j,\, \sigma(l)>\sigma(m)\big\} = \frac{n-2}{2}(n-2)!. 
\end{align*}
4. Like in 1., but we use $\#\big\{\sigma \in \mathfrak{S}_n\ \big|\ \sigma(1)=j,\, \sigma(i) \neq 1,\, j>\sigma(i)\big\} = (j-2)(n-2)!$,
\begin{align*}
& \text{for}\ m \in [n] \setminus \{1,i\},\ \#\big\{\sigma \in \mathfrak{S}_n\ \big|\ \sigma(1) =j,\, \sigma(i) \neq 1,\, j>\sigma(m)\big\} = \big((j-1)(n-4) + j\big)(n-3)!,  \\
& \text{for}\ l \in [i-1] \setminus \{1\},\ \#\big\{\sigma \in \mathfrak{S}_n\ \big|\ \sigma(1)=j,\, \sigma(i) \neq 1,\, \sigma(l)>\sigma(i)\big\} = \binom{n-2}{2}(n-3)!,  \\
& \text{for}\ m \in [n] \setminus [i],\ \#\big\{\sigma \in \mathfrak{S}_n\ \big|\ \sigma(1) =j,\, \sigma(i) \neq 1,\, \sigma(i)>\sigma(m)\big\} = \binom{n-1}{2}(n-3)!.
\end{align*}
\end{proof}

\begin{lemma}  \label{LeX}
Let $n \geq 2$, $\lambda_1, \dots, \lambda_n \in \mathbb{R}$, and $\mathsf{x}_1, \dots, \mathsf{x}_n \in \mathbb{R}[y_{1,2}, \dots, y_{n,n-1}]$ such that $\mathsf{x}_1 \neq 0$. Consider the matrix $\mathsf{X} = (\lambda_i \mathsf{x}_j)_{i,j \in [n]}$. Then, $$\mathrm{Sp}(\mathsf{X}) = \Big\{\sum_{i \in [n]} \lambda_i \mathsf{x}_i,\, 0\Big\} \quad \text{with} \quad \mathrm{m}_{\mathsf{X}}\Big(\sum_{i \in [n]} \lambda_i \mathsf{x}_i\Big) = 1,\ \mathrm{m}_{\mathsf{X}}(0) = n-1.$$
\end{lemma}

\begin{proof}
Let $\{\mathbf{e}_i\}_{i \in [n]}$ be the $n$-dimensional standard basis. The eigenspace of $\mathsf{X}$ associated to the eigenvalue $0$ is $\langle \mathsf{x}_i \mathbf{e}_1 - \mathsf{x}_1 \mathbf{e}_i \rangle_{i \in [n] \setminus \{1\}}$. We deduce that $\mathrm{m}_{\mathsf{X}}(0) = n-1$. Hence, the characteristic polynomial of $\mathsf{X}$ is $t^{n-1}(t- \mathrm{tr}\,\mathsf{X})$.
\end{proof}

\begin{lemma}  \label{Leg}
For $n \geq 4$, define
$$\mathsf{g} := \sum_{i \in [n] \setminus \{1\}} \frac{n-2i+1}{2}(n-2)! \Big(2y_{1,i} + \sum_{m \in [n] \setminus \{1,i\}}y_{1,m} + \sum_{l \in [i-1] \setminus \{1\}}y_{l,i} - \sum_{m \in [n] \setminus [i]}y_{i,m}\Big) \in \mathbb{R}[y_{1,2}, \dots, y_{n,n-1}].$$
Then, $\displaystyle \mathsf{g} = -(n-2)! \sum_{\substack{(i,j) \in [n]^2 \\ i<j}} (j-i)y_{i,j}$.
\end{lemma}

\begin{proof}
For $i \in [n] \setminus \{1\}$, we have
\begin{align*}
[y_{1,i}]\mathsf{g} & = \Big(n-2i+1 + \sum_{j \in [n] \setminus \{1,i\}} \frac{n-2j+1}{2}\Big)(n-2)!  \\
& = \Big(1-i + \sum_{j \in [n]} \frac{n-2j+1}{2}\Big)(n-2)!  \\
& = -(n-2)!(i-1).
\end{align*}
For $i,j \in [n] \setminus \{1\}$ with $i<j$, we have
$$[y_{i,j}]\mathsf{g} = \frac{n-2i+1}{2}(n-2)! - \frac{n-2j+1}{2}(n-2)! = -(n-2)!(j-i).$$
\end{proof}

\begin{proposition}  \label{Prinv}
For $n \geq 4$, $\mathscr{X}_{S^{(n-1,1)}}\big(\mathsf{i}_y(n) \big)$ is diagonalizable, and
$$\mathrm{Sp}\Big(\mathscr{X}_{S^{(n-1,1)}}\big(\mathsf{i}_y(n)\big)\Big) = \big\{-(n-2)! \sum_{\substack{(i,j) \in [n]^2 \\ i<j}} (j-i)y_{i,j},\, 0\big\}$$ 
\begin{itemize}
	\item with $\displaystyle \mathrm{m}_{\mathscr{X}_{S^{(n-1,1)}}\big(\mathsf{i}_y(n)\big)}\big(-(n-2)! \sum_{\substack{(i,j) \in [n]^2 \\ i<j}} (j-i)y_{i,j}\big) = 1$,
	\item and $\mathrm{m}_{\mathscr{X}_{S^{(n-1,1)}}\big(\mathsf{i}_y(n)\big)}(0) = n-2$.
\end{itemize} 
\end{proposition}

\begin{proof}
Using Lemma~\ref{LeYij} and Lemma~\ref{Leg}, we obtain
\begin{align*}
[\mathbf{1}-\mathbf{j}]\mathsf{i}_z(n)(\mathbf{1}-\mathbf{i}) & = \sum_{\substack{\sigma \in \mathfrak{S}_n \\ \sigma(1)=1,\, \sigma(i)=j}} \mathrm{inv}_y(\sigma) - \sum_{\substack{\sigma \in \mathfrak{S}_n \\ \sigma(1)=j,\, \sigma(i)=1}} \mathrm{inv}_y(\sigma) + \sum_{\substack{\sigma \in \mathfrak{S}_n \\ \sigma(1) \neq 1,\, \sigma(i)=j}} \mathrm{inv}_y(\sigma) - \sum_{\substack{\sigma \in \mathfrak{S}_n \\ \sigma(1)=j,\, \sigma(i) \neq 1}} \mathrm{inv}_y(\sigma)  \\
& = \frac{n-2j+1}{2}(n-2)! \Big(2y_{1,i} + \sum_{m \in [n] \setminus \{1,i\}}y_{1,m} + \sum_{l \in [i-1] \setminus \{1\}}y_{l,i} - \sum_{m \in [n] \setminus [i]}y_{i,m}\Big)  \\
& = -(n-2)! \sum_{\substack{(i,j) \in [n]^2 \\ i<j}} (j-i)y_{i,j}.
\end{align*}
Observe that the form of $\mathscr{X}_{S^{(n-1,1)}}\big(\mathsf{i}_y(n) \big)$ is similar to that in Lemma~\ref{LeX}. Thus, $\mathrm{Sp}\Big(\mathscr{X}_{S^{(n-1,1)}}\big(\mathsf{i}_y(n)\big)\Big)$ is composed of $\displaystyle -(n-2)! \sum_{\substack{(i,j) \in [n]^2 \\ i<j}} (j-i)y_{i,j}$ and $0$ whose multiplicities are respectively $1$ and $n-2$. Finally, since the regular representation of $\mathsf{i}_y(n)$ is diagonalizable \cite[Theorem~1.4]{Ra1}, then $\mathscr{X}_{S^{(n-1,1)}}\big(\mathsf{i}_y(n) \big)$ is diagonalizable.
\end{proof}

\noindent We can now proceed to the proof of Theorem~\ref{ThSp}.

\begin{proof}
From Proposition~\ref{Prfix2} and Proposition~\ref{Prinv}, we deduce that
\begin{align*}
\mathrm{diag}\,\mathscr{X}_{S^{(n-1,1)}}\big(\mathsf{i}_y(n) + \mathsf{f}_z(n)\big) & = \mathrm{diag}\,\mathscr{X}_{S^{(n-1,1)}}\big(\mathsf{i}_y(n)\big) + \mathrm{diag}\,\mathscr{X}_{S^{(n-1,1)}}\big(\mathsf{f}_z(n)\big)  \\
& = -(n-2)! \sum_{\substack{(i,j) \in [n]^2 \\ i<j}} (j-i)y_{i,j} \mathsf{I}_1 \oplus 0 \mathsf{I}_{n-2} \, + \, n(n-2)!z\mathsf{I}_{n-1} \\
& = (n-2)! \big(nz - \sum_{\substack{(i,j) \in [n]^2 \\ i<j}} (j-i)y_{i,j}\big)\mathsf{I}_1 \, + \, n(n-2)!z\mathsf{I}_{n-2}. 
\end{align*}
\end{proof}

\section{Proof of Theorem~\ref{ThDIF}}  \label{SeTh}

\noindent Let $\mathsf{I}(n) := \big(\mathrm{inv}_y(\sigma \tau^{-1})\big)_{\sigma, \tau \in \mathfrak{S}_n}.$ From Theorem~\ref{ThF} and \cite[Theorem~1.4]{Ra1}, we know that $\mathsf{IF}(n)$ is diagonalizable, and $$\mathrm{diag}\,\mathsf{IF}(n) = \mathrm{diag}\,\mathsf{F}(n) + \mathrm{diag}\,\mathsf{I}(n).$$

\noindent Using the theorem of Perron-Frobenius for multinomial \cite[Proposition~2.1]{Ra1}, we get $$\frac{n!}{2}\sum_{\substack{(i,j) \in [n]^2 \\ i<j}} y_{i,j} + n!z \in \mathrm{Sp}\big(\mathsf{IF}(n)\big) \quad \text{with} \quad \mathrm{m}_{\mathsf{IF}(n)}\Big(\frac{n!}{2} \sum_{\substack{(i,j) \in [n]^2 \\ i<j}} y_{i,j} + n!z\Big) = 1.$$

\noindent The theorem of Maschke allows to state that 
$$\mathscr{X}_{\mathbb{R}[y_{1,2}, \dots, y_{n,n-1},z][\mathfrak{S}_n]}\big(\mathsf{i}_y(n) + \mathsf{f}_z(n)\big) = \bigoplus_{\lambda \in \mathrm{Par}(n)}\overbrace{\mathscr{X}_{S^{\lambda}}\big(\mathsf{i}_y(n) + \mathsf{f}_z(n)\big) \oplus \dots \oplus \mathscr{X}_{S^{\lambda}}\big(\mathsf{i}_y(n) + \mathsf{f}_z(n)\big)}^{d_{\lambda}}.$$
Combining that with Theorem~\ref{ThSp}, and comparing with $\mathrm{m}_{\mathsf{F}(n)}\big(n(n-2)!z\big) = (n-1)^2$ and $\mathrm{m}_{\mathsf{I}(n)}\Big(-(n-2)!\sum_{\substack{(i,j) \in [n]^2 \\ i<j}} (j-i)y_{i,j} \Big) = n-1$, we deduce that
\begin{align*}
& (n-2)! \big(nz - \sum_{\substack{(i,j) \in [n]^2 \\ i<j}} (j-i)y_{i,j}\big),\, n(n-2)!z \in \mathrm{Sp}\big(\mathsf{IF}(n)\big) \ \text{with} \\
& \mathrm{m}_{\mathsf{IF}(n)}\Big((n-2)! \big(nz - \sum_{\substack{(i,j) \in [n]^2 \\ i<j}} (j-i)y_{i,j}\big)\Big) = n-1\ \text{and}\ \mathrm{m}_{\mathsf{IF}(n)}\big(n(n-2)!z\big) = (n-2)(n-1). 
\end{align*}
We deduce from Theorem~\ref{ThF} and \cite[Theorem~1.4]{Ra1} that $-(n-3)! \sum_{\substack{(i,j) \in [n]^2 \\ i<j}} \big(n -2(j-i)\big) y_{i,j}$ and $0$ are finally the last eigenvalues of $\mathsf{IF}(n)$, and their multiplicities are
$$\mathrm{m}_{\mathsf{IF}(n)}\Big(-(n-3)! \sum_{\substack{(i,j) \in [n]^2 \\ i<j}} \big(n -2(j-i)\big) y_{i,j}\Big) = \binom{n-1}{2}\ \text{and}\ \mathrm{m}_{\mathsf{IF}(n)}(0) = n! -\frac{n}{2}(3n-7) -3.$$
\begin{flushright} $\blacksquare$ \end{flushright}

\bibliographystyle{abbrvnat}

\end{document}